\documentclass[12pt]{article}
\usepackage{amsmath,amsthm,amsfonts,fancyhdr}
\usepackage{graphicx,enumerate,relsize}

\newcommand \m{\mathfrak M(\lambda,M)}

\newcommand{\p}{\Phi}

\newcommand{\dd} {\delta}

\newcommand{\g}{\lambda}

\newtheorem*{thm*}{ Theorem, Hellinger }
\newtheorem*{thm2*}{ Theorem, Kostyuchenko-Mirzoev, 1998}
\newtheorem*{thm3*} {Theorem, Berg-Dur\'an}

\newtheorem{theorem}{Theorem}
\newtheorem*{theorem*}{Theorem}
\newtheorem*{corollary*}{Corollary}
\newtheorem*{proposition}{Proposition}
\newtheorem{lemma}{Lemma}
\newcounter{cont}

\begin{document}
\title {Miura-like transformations between Bogoyavlensky lattices and inverse spectral problems for band operators.}
\author{ Andrey Osipov  \\
Federal State Institution \\
``Scientific-Research Institute for System Analysis\\
of the Russian Academy of Sciences" (SRISA)\\
Nakhimovskii pr. 36-1, Moscow, 117218 \\ Russia
\\
email: {osipa68@yahoo.com} }


\date{}
 \maketitle
\begin{abstract}
We consider semi-infinite and finite Bogoyavlensky lattices
\begin{eqnarray*}
\overset\cdot a_i&=&a_i\left(\prod_{j=1}^{p}a_{i+j}-\prod_{j=1}^{p}a_{i-j}\right),\\
\overset\cdot b_i&=&b_i\left(\sum_{j=1}^{p}
b_{i+j}-\sum_{j=1}^{p}b_{i-j}\right),
\end{eqnarray*}
for some $p\ge 1,$ and Miura-like transformations between these
systems, defined for $p\ge 2$. Both lattices are integrable (via
Lax pair formalism) by the inverse spectral problem method for
band operators, i. e. operators generated by (possibly infinite)
band matrices. The key role in this method is played by the
moments of the Weyl matrix of the corresponding band operator and
their evolution in time. We find a description of the
above-mentioned transformations in terms of these moments and
apply this result to study the finite Bogoyavlensky lattices and
in particular their first integrals.
\end{abstract}
\textit{Key Words:}  Difference operators, Inverse spectral
problems, Nonlinear lattices, Miura transformations \newline
\textit{Mathematics Subject Classification} 2010: 47B36, 37K10,
37K15

\section{Introduction}
Since the pioneering work of C. Gardner, J. Greene, M. Kruskal and
R. Miura \cite{ggkm}, the integration of nonlinear equations by
using various inverse problems methods is among the main topics in
modern mathematical physics. This integration task has inspired
the development of many aspects of the theory of differential and
difference operators (the inverse problems for the latter can be
considered as part of operator theory) as well as the areas of
mathematics related to inverse problems. For almost half a
century, see \cite{km}-\cite{mos}, the inverse spectral problems
for difference operators have been applied for integration of
certain nonlinear dynamical systems called nonlinear lattices. As
an example of such an application, we mention the work by
Berezanski \cite{bert}, where the initial boundary value problem
for the semi-infinite Toda lattice was solved by using the
classical inverse spectral problem for Jacobi operators. Further
on, this inverse spectral problem method was developed aiming to
cover wider classes of nonlinear lattices, see e. g.
\cite{geh,yur,os4}. Such activity has also inspired the
development of the areas of the function theory connected with the
study of difference operators \cite{apt,akv,esp,sor}.

Turning to nonlinear integrable equations, note that an important
role in their study is played by various Miura-like
transformations which relate the equations and their solutions.
For example, a discrete Miura transformation between Kac-van
Moerbeke (Volterra) and Toda lattices allows one to derive the
$N-$ soliton solutions for both such systems starting from the
trivial ones \cite{gez}. Also, as noted in \cite{dam1,dam2}, such
transformation maps onto each other the first integrals,
Hamiltonians, Poisson brackets, master symmetries of these two
systems (both of them are rich in these objects of interest of the
integrable systems theory). In \cite{osrjnd,osconop} an easy
description of this transformation in terms of the inverse
spectral data for Jacobi operators which appear in the Lax
representation for both lattices was obtained. Note that it was
recently found in \cite{osconop,osspm} that such transformations
can be applied to the study of self-adjointness of Jacobi
operators. The latter result illustrates that the links between
the theory of nonlinear integrable equations and the operator
theory are not based entirely upon the above-mentioned inverse
problems.

Here we study similar transformations between Bogoyavlensky
lattices  \cite{bog,sur,zh} and obtain their description via the
inverse spectral data for band operators (the latter may also be
regarded as high order difference ones \cite{yur}) which arise in
the Lax pairs of such systems. We also apply this result to the
case of finite Bogoyavlensky lattices. In particular, we show how
our findings can be helpful for obtaining some ``non-standard"
first integrals of these systems.

The paper is organized as follows. In the next section the
semi-infinite Bogoyavlensky lattices and their integration by
means of the inverse spectral problem for band operators is
considered. In Section 3 the above-mentioned description of
Miura-like transformations between these lattices is obtained.
Finally, in Section 4 we consider the finite lattices and show how
the previous result can be applied to study of their first
integrals.
\section{Bogoyavlensky lattices. Inverse problem method}
Consider the Cauchy problem for two nonlinear dynamical systems in
the class of bounded solutions in the semi-infinite case:
\begin{eqnarray}
\overset \cdot
a_i&=&a_i\left(\prod_{j=1}^{r}a_{i+j}-\prod_{j=1}^{r}a_{i-j}\right ),\quad \text{for a fixed}\quad r\in \mathbb{N};
\label{prob}\\
\overset \cdot
b_i&=&b_i\left(\sum_{j=1}^{q}b_{i+j}-\sum_{j=1}^{q}b_{i-j}\right ),\quad \text{for a fixed}\quad q\in \mathbb{N};\label{sumb}\\
 i \in \mathbb{Z}_+,\; a_i&=&a_i(t),\,\; b_i=b_i(t),\quad a_i,b_i \in \mathbb{C},\quad t\in [0,T),\quad 0<T\le \infty;\nonumber \\
a_i,b_i &\ne& 0,\quad
(a_i(t))_{i=0}^\infty,\,(b_i(t))_{i=0}^\infty \in l_{\infty};\quad
b_{l}=a_{l}=0  \quad\text{for} \quad l<0. \nonumber
\end{eqnarray}
Both systems \eqref{prob} and \eqref{sumb} (in the infinite case,
i. e. when $i\in \mathbb{Z}\,$) were introduced by Bogoyavlensky,
see \cite{bog} and references thereafter, where it was shown that
they can be regarded as discrete versions of the Korteweg-de Vries
equation.

 The system \eqref{prob} admits the Lax
representation $\overset \cdot L = [L,A] $ with the infinite
matrices
\begin{eqnarray}
\label{probl}
L&=&L1=L1(t)=(L1_{ij})_{i,j=0}^{\infty}= \nonumber \\&=&\begin{pmatrix} 0&0&\dots& 0_{0\,r-1}&1& 0&0&0&\dots \\
           a_0&0&0&\dots&0_{1 r}&1&0&0&\dots\\
           0&a_1&0&0&\dots&0_{2\,r+1}&1&0&\dots\\
           0&0&a_2&0&0&\dots&0&1&\dots\\
            \vdots&\vdots&\vdots&\ddots&\vdots&\vdots&\vdots&\vdots&\ddots
            \end{pmatrix},
\end{eqnarray}

\begin{equation}
\label{proba}
 A=A1=A1(t)=(A1_{ij})_{i,j=0}^{\infty}=\begin{pmatrix} 0& 0& 0&0&0&0&\dots\\
           \vdots&\vdots&\vdots&\vdots&\vdots&\vdots&\dots\\
           0_{r 0}&0&0&0&0&0&\dots\\
            a_0\cdots a_r&0&0&0&0&0&\dots\\
            0&a_1\cdots a_{r+1}&0&0&0&0&\dots\\
            \vdots&\vdots&\ddots&\vdots&\vdots&\vdots&\dots
            \end{pmatrix}.
\end{equation}

and the Lax pair for the system \eqref{sumb} can be defined as
follows
\begin{eqnarray}
\label{sumbl} L&=&L2=L2(t)=(L2_{ij})_{i_,j=0}^{\infty}= \nonumber
\\&=&\begin{pmatrix}0&1& 0&0&0&0&0&\dots\\
           \vdots&\dots&\ddots&\dots&\dots&\dots&\dots&\dots\\
           0_{q-1\,0}&\dots&0&1_{q-1\, q}&0&0&0&\dots\\
           b_0&0_{q 1}&\dots&0&1_{q \,q+1}&0&0&\dots\\
           0&b_1&0&\dots&0&1&0&\dots\\
            \vdots&\vdots&\ddots&\vdots&\vdots&\vdots&\ddots&\dots
            \end{pmatrix}; \nonumber \\
\end{eqnarray}
\begin{eqnarray}
\label{sumba}
 A&=&A2=A2(t)=(A2_{ij})_{i,j=0}^{\infty}= \nonumber
 \\&=&-\begin{pmatrix} b_0& 0&\dots&0_{0\, q}&1&0&0&0&0&\dots\\
           0&b_0+b_1&0&0&0_{1\,q+1}&1&0&0&0&\dots\\
           \vdots&\dots&\ddots&\vdots&\vdots&\vdots&\ddots&\vdots&\vdots&\vdots\\
            0_{q\,0}&\dots&0_{q \,q-1}&\sum_{i=0}^q b_i&0&\dots&0_{q\, 2q}&1&0&\dots\\
             0&0&\dots&0_{q+1\,q}&\sum_{i=1}^{q+1} b_i&0&\dots&0&1&\dots\\
            \vdots&\vdots&\vdots&\vdots&\ddots&\vdots&\vdots&\vdots&\vdots&\ddots
            \end{pmatrix}. \nonumber \\
\end{eqnarray}

Both $L1$ and $L2$ are special cases of the matrix $M=(a_{i
k})_{i,k=0}^\infty$ with the following elements:
\begin{equation}
\label{Mcond}
 \gathered
a_{i k} \in \mathbb C, \quad a_{i k}=0, \quad k>i+r ,\quad  i>k+q ,\\
a_{i\, i+r}=1, \quad a_{i+q\, i}\ne 0,\quad i \ge 0;
\endgathered
\end{equation}
 i. e.
\begin{equation}
\label{matM}
M=\begin{pmatrix} a_{0 0}&\dots&1_{0\,r}&0&0&\dots&\dots&\dots\\
    a_{1 0}&a_{1 1}&\dots&1&0&\dots&\dots&\dots\\
   \vdots&\vdots&\vdots&\vdots&\ddots&\vdots&\vdots&\dots\\
a_{q\,0}&a_{q\,1}&a_{q\,2}&\dots&\dots&1_{q\,q+r}&0&\dots\\
  0&a_{q+1\,1}&a_{q+1\,2}&a_{q+1\,3}&\dots&\dots&1&\dots\\
0&0&\ddots&\vdots&\vdots&\vdots&\vdots&\ddots
\end{pmatrix};
\end{equation}
so $M$ is an infinite non-symmetric band matrix which consists of
$q+r+1$ (possibly) non-zero diagonals. Denote as $M$ any matrix of
this structure. As we see, the matrix $L1$ fits into the case of
$M$ with $q=1,$ whereas for $L2$ the corresponding case is $r=1$.

Denote by $l^2 [0,\infty) $ the Hilbert space of the complex
sequences $y=(y_n)_{n=0}^{\infty}$ such that $\sum_{n=0}^\infty
|y_n|^2 < \infty $, with the inner product $ (y,z)=
\sum_{n=0}^\infty y_n \bar z_n$. Also denote by $
\{e_n\}_{n=0}^{\infty} $ its standard orthonormal basis. We
identify the matrix $M$ with the operator defined as the closure
of the operator acting on the dense set of finite vectors from
$l^2[0,\infty)$, where its action is described via matrix calculus
(and keep the same notation $M$ for this operator).

Now, consider in brief the inverse spectral problem for the
operators $M;$ for its full description see e. g. \cite{osrjmp}.
First, for $\g \in \Omega(M),$ where $\Omega(M)$ is the resolvent
set of  $M,$ we define the following functions named the Weyl
solutions of $M$ \cite{yur, yurm, osrjmp}:
\begin{eqnarray}
\label{weyls}
\gathered
\p(\g)=(\p_i(\g))_{i=0}^\infty, \quad \p_i(\g)=(\p_i^1(\g),\dots,\p_i^q(\g)),\\
\p_i^n(\g):=(R_{\g}e_{n-1})_i, \quad n=1,\dots,q,\\ R_{\g}=(\g
E-M)^{-1} -\text{resolvent of} \quad M;\quad E - \text{identity
operator}.
\endgathered\
\end{eqnarray}
In other words, $\p(\g)=(R_\g e_0,\dots,R_\g e_{q-1})$. Also
define the Weyl matrix for $M$ as follows:
\begin{eqnarray}
\label{weylm} \m =(\mathfrak M_{m,n}(\g,M))_{m=1,\dots ,r}^{n=1,
\dots, q};\qquad \g \in \Omega(A);\nonumber \\ \mathfrak
M_{m,n}(\g,M):=\Phi_{m-1}^n(\g)=(R_{\g}e_{n-1})_{m-1}=(R_{\g}e_{n-1},e_{m-1}).
\end{eqnarray}
 For  $q=1, r=1$,
the matrix   $\m$ coincides with the Weyl function  for the
corresponding tridiagonal matrix of Jacobi type and if $J$ is the
classical Jacobi operator, then $\mathfrak M(\lambda,J)$ is the
Stieltjes transform of its spectral measure, see
\cite{ah,bert,bk,osrjnd,osconop}. We also introduce the following
system of formal power series with the parameter  $ \g \in \mathbb
C\backslash \{0\} $:
\begin{equation}
\label{Wps} \gathered
W(\g)=(W_{m,n} (\g))_{m=1\dots r}^{n=1\dots,q},\\
W_{m,n}(\g)=\sum_{k=0}^{\infty} \frac{S^{m,n}_{k}}{\g^{k+1}};\quad
S^{m,n}_{k}=(M^k)_{m-1,n-1};
\endgathered
\end{equation}
where $(M^k)$ is the $k$ - th power of the matrix $M$. If the
operator $M$ is bounded (when $\sup_{i,j}\mid a_{i,i+j} \mid <
\infty$, which holds for $L1$ and $L2$), then the Neumann formula
for its resolvent is valid for $\mid \g \mid > \Vert M \Vert $
and, as follows from \eqref{weyls},
\begin{equation}
\label{weylneum}
\p_i^n(\g)=\sum_{k=0}^\infty
\frac{(M^k)_{i,n-1}}{\g^{k+1}}
\end{equation}
for all $i$ and $m$. Therefore, for the bounded operators $M$, the
functions $ \mathfrak M_{m,n}(\g)=\p^n_{m-1}(\g) $ are holomorphic
at infinity and $ \mathfrak M_{m,n}(\g)=W_{m,n}(\g) $ in its
neighborhood. This allows us to define, in the general case, the
asymptotic expansion of the Weyl matrix of $ M $ at infinity, as
the matrix \eqref{Wps}:
\begin{equation*}
\label{winf}
 \mathfrak M_{\infty}(\g,M) \overset{def}= W(\g).
\end{equation*}

Now we introduce the object which plays the key role in the
considered inverse spectral problem method. Namely, the sequence
\begin{equation*}
\gathered S(\m)=(S_k)_{k=0}^\infty;\quad
S_k=\begin{pmatrix} &S_{k}^{1,1}&\dots &S_{k}^{1,q}\\
               &\vdots&&\vdots\\
     &S_{k}^{r,1}&\dots &S_{k}^{r,q}
     \end{pmatrix};\\
\endgathered
\end{equation*}
where $S_{k}^{m,n}$ for $m=1,\dots,r,\;n=1,\dots,q $ are defined
by \eqref{Wps} is called the moment sequence of  the Weyl matrix
of $M$. The values $ S_{k}^{m,n} $ are called the moments of $\m$.

As in the tridiagonal matrix case \cite{osrjnd}, our inverse
spectral problem admits the following formalization: given
$S(\m),$ find $M$.

The elements of $M$ can be recovered from $S(\m)$ in the recurrent
manner starting from $a_{i+q,i}$, see \cite{osrjmp}, and the
latter are obtained from the following formulas:

\begin{equation}
\label{lefta}
a_{i+q,i}=\frac{\Delta_{i+q}\Delta_{i-1}}{\Delta_{i+q-1}\Delta_{i}},\quad
i \ge 0;
\end{equation}
where
\begin{eqnarray}
\label{delt}
 \Delta_{k}=\operatorname{det}(H_k) \quad\text{for}\;H_k=(\alpha_{i,j})_{i,j=0}^{k},\;
\Delta_{-1}=1;\quad \text{and}\quad \alpha_{i,j}=S^{m1,n1}_{k1+k2}\\
k1=\left\lfloor\frac{i}{r}\right\rfloor,\quad
 k2=\left\lfloor\frac{j}{q}\right\rfloor,\quad
 m1=\operatorname{rem}(i,r)+1,\quad n1=\text{rem}(j,q)+1.\nonumber
\end{eqnarray}
(here and in what follows $\operatorname{rem}(a,b)$ denote the
remainder after division of $a$ by $b$.) Also set
\begin{equation}
\label{hinf} H=(\alpha_{i,j})_{i,j=0}^\infty.
\end{equation}

 Another important issue is a solvability criterion
for the considered inverse problem method for $M$ in terms of the
moment sequence of its Weyl matrix. It can be formulated as
follows.
\begin{theorem}
\label{thcrit}
 The sequence
\begin{equation*}
\gathered S(\m)=(S_k)_{k=0}^\infty;\;
S_k=\begin{pmatrix}&S_{k}^{1,1}&\dots &S_{k}^{1,q}\\
               &\vdots&&\vdots\\
     &S_{k}^{r,1}&\dots &S_{k}^{r,q}
     \end{pmatrix},\\
     S_{k}^{m,n}\in \mathbb{C},\quad m=1,\dots ,r,\;n=1,\dots ,q,
     \endgathered
\end{equation*}
is the moment sequence of  the Weyl matrix of the operator $M,$ if
and only if the following conditions hold:
\begin{enumerate}[(i)]
\item  $$ S_l^{m,n}=\dd_{m+lr,n},\qquad n>lr,\qquad
l=0,\dots,\left[\frac{q}{r}\right]\,$$ - normalization condition;

\item for every $k\ge 0,\quad \Delta_k\ne 0,$ where $\Delta_k$ are
the determinants defined according to \eqref{delt}.
 \setcounter{cont}{\value{enumi}}
\end{enumerate}
\end{theorem}
For  $r=1, q=2,$ the proof is given in \cite{osf}; the case of
arbitrary $r$ and $q$ can be proved similarly (see also
\cite{os2,yur,yurm} where a similar criterion was established for
another classes of band operators). Note that from the proof
follows that there is a one-to-one correspondence between the
matrices $M$ and the sequences satisfying the above conditions
$(i)-(ii)$. Obviously, the condition $ (ii) $ implies that $$
\operatorname{rank} H=\infty,\, $$ where $H$ is defined in
\eqref{hinf}.

Using the Theorem \ref{thcrit} (condition $(i)$) and \eqref{lefta}
and applying induction on $k$ we establish the following result:
\begin{lemma}
\label{lemdet} The determinants $\Delta_k$ \eqref{delt} are
calculated by the formulas
\begin{eqnarray}
\label{detform} \Delta_0&=&\dots=\Delta_{q-1}=1, \nonumber \\
\Delta_{i+q}&=&a_{i+q,\, i}\cdots
a_{i+1,\,i-q+1}(a_{i,\,i-q}\cdots
a_{i-q+1,\,i-2q+1})^2\times\cdots
\nonumber\\&\times&(a_{i-(h-3)q,\,i-(h-2)q}\cdots
a_{i-(h-2)q+1,\,i-(h-1)q+1})^{h-1}\times\nonumber\\&\times&(a_{i-(h-2)q,\,
i-(h-1)q}\cdots a_{q,\, 0})^{h};\nonumber \\ i &\ge& 0,\quad
h=\left \lfloor \frac{i+q}{q}\right \rfloor;\quad a_{i,k}=1\quad
\text{for} \quad k<0.
\end{eqnarray}
\end{lemma}
We now turn back to the systems \eqref{prob}-\eqref{sumb} and,
respectively, to the matrices $L1$ and $L2$. As follows from the
above, for the eponymous operators the moment sequences are
\begin{equation}
\label{sL1}
 S(\mathfrak
M(\lambda,L1))(t)=\begin{pmatrix}S^{1,1}_k\\\vdots\\S^{r,1}_k\end{pmatrix}_{k=0}^\infty;\quad
\text{where}\quad S_k^{m,1}=(L1^k)_{m-1,\, 0};
\end{equation}
and, respectively,
\begin{equation}
\label{sL2} S(\mathfrak
M(\lambda,L2))(t)=(S^{1,1}_k,\dots,S^{1,q}_k)_{k=0}^\infty;\quad
\text{where}\quad S_k^{1,n}=(L2^{\,k})_{0, \,n-1}.
\end{equation}

 Both $L1$ and $L2$ correspond to the case of
``sparse'' matrices $M,$ i. e. such ones that in addition to
\eqref{Mcond} the conditions
\begin{equation}
\label{spar}
 a_{i,k}=0 \quad \text{for} \quad k=i-q+1,\dots,i+r-1
\end{equation}
are fulfilled. In terms of $S(\mathfrak M(\lambda,L1))(t)$ and
$S(\mathfrak M(\lambda,L2))(t)$ this property implies that
\begin{eqnarray}
\label{sparl1}  S_k^{m,1}&=&0, \quad m=1,\dots,r, \qquad k \ne
m-1+(r+1)k'; \\
\label{sparl2} S_k^{1,n}&=&0, \quad n=1,\dots,q, \qquad k \ne
n-1+(q+1)k', \quad k'\in \mathbb{Z}_+.
\end{eqnarray}
For the complete proof of sparsity criterion for the matrices $M$
in terms of $S(\m)$ see \cite{osrjmp}, Theorem 2. Another property
of the moments of $S(\mathfrak M(\lambda,L1))(t)$ and $S(\mathfrak
M(\lambda,L2))(t)$ can be established directly using
\eqref{proba},\eqref{sumba} and \eqref{sL1}-\eqref{sL2}, namely:
\begin{lemma}
\label{lemmom}
 For the elements of $S(\mathfrak M(\lambda,L1))(t)$ and $S(\mathfrak
M(\lambda,L2))(t)$ defined according to \eqref{sL1}-\eqref{sL2}
\begin{equation}
\label{moml1l2} S_{m-1}^{m,1}=a_0\dots a_{m-2},\; m\ge 2;\;
S_{r+1}^{1,1}=a_0\dots a_{r-1};\quad S_{n-1}^{1,n}=1,\;
S_{n+q}^{1,n}=\sum_{l=0}^{n-1} b_l;
\end{equation}
\end{lemma}
($S_{0}^{1,1}=1\,$ according to the condition $(i)$ of the Theorem
\ref{thcrit}). Our next aim is to find the evolution equations for
the moments. First note that since $ \overset \cdot R_{\g}(t)=-
R_{\g}(t) (\g E\overset \cdot - L(t))R_{\g}(t),$ it follows from
the Lax equation that
\begin{equation}
\label{laxres}
\overset \cdot R_{\g}(t)= R_{\g}(t)\left(A(t)(\g I
- L(t)) - (\g I - L(t))A(t)\right)R_{\g}(t)=[R_{\g}(t), A(t)].
\end{equation}
Let $R^{1}_{\g}$ and $R^{2}_{\g}$ be the resolvents of $L1$ and
$L2$ respectively. Denote by $(R^{1}_{i,j})_{i,j=0}^\infty$ and
$(R^{2}_{i,j})_{i,j=0}^\infty$ their matrix representations in the
basis $\{e_n\}$. As follows from \eqref{weylm}, $\; \mathfrak
M_{m,1}(\g,L1)=\mathfrak
M_{m,1}(\g,L1)(t)=R^{1}_{m-1,0},\;\mathfrak
M_{1,n}(\g,L2)=\mathfrak M_{1,n}(\g,L2)(t)=R^{2}_{0,n-1}$.Then,
using \eqref{laxres} and \eqref{probl}-\eqref{proba} we find that
\begin{equation}
\label{r1i0} \overset  \cdot  {\;\;R^1_{i,0}}=a_0\dots a_r
R_{i,r+1}^1;\quad i=0,\dots, r-1.
\end{equation}
From the identity
\begin{equation}
\label{matidl1} R^{1}_{\g}(\g E - L1)=E,
\end{equation}
written in the matrix form we get the following chain of relations
\begin{eqnarray*}
-R^{1}_{0,0}+\g R^{1}_{0,r}-a_r R^{1}_{0,r+1}&=&0,\\
\g R^{1}_{0,r-1}-a_{r-1} R^{1}_{0,r}&=&0,\\
&\vdots& \\
\g R^{1}_{0,1}-a_1R^{1}_{0,2}&=&0, \\
\g R^{1}_{0,0}-a_0R^{1}_{0,1}&=&1;
\end{eqnarray*}
from which we obtain
\begin{equation*}
a_r R^{1}_{0,r+1} = \frac{\g^{q+1}R^{1}_{0,0}-\g^q}{a_0\dots
a_{r-1}}-R^{1}_{0,0}.
\end{equation*}
Substituting the latter into \eqref{r1i0} for $\,i=0\,$, we find
the equation for \newline$\mathfrak M_{1,1}(\lambda,L1):$
\begin{equation*}
\overset \cdot{\quad \mathfrak
M_{1,1}}(\lambda,L1)=\g^{r+1}\mathfrak M_{1,1}(\lambda,L1) - \g^r
-a_0\dots a_{r-1}\mathfrak M_{1,1}(\lambda,L1).
\end{equation*}
Using \eqref{Wps} - \eqref{weylneum} and
\eqref{sL1},\eqref{moml1l2}, we get the corresponding equation for
the moments
\begin{equation*}
\overset \cdot
{\;\;S_k^{1,1}}=S_{k+r+1}^{1,1}-S_{r+1}^{1,1}S_k^{1,1}.
\end{equation*}
For $i=1,\dots,r-1$ we derive from \eqref{matidl1} the following
relations
\begin{eqnarray*}
-R^{1}_{i,0}+\g R^{1}_{i,r}-a_r R^{1}_{i,r+1}&=&0,\\
\g R^{1}_{i,r-1}-a_{r-1} R^{1}_{i,r}&=&0,\\
&\vdots& \\
 \g R^{1}_{i,0}-a_0R^{1}_{i,1}&=&0;
\end{eqnarray*}
from which we have
\begin{equation*}
a_r R^{1}_{i,r+1} = \frac{\g^{q+1}R^{1}_{i,0}}{a_0\dots
a_{r-1}}-R^{1}_{i,0};
\end{equation*}
so the substitution of the latter into \eqref{r1i0} leads to the
following equations
\begin{eqnarray*}
&&\overset\cdot{\qquad R^1_{m-1,0}}=\overset \cdot{\quad \mathfrak
M_{m,1}}(\lambda,L1)=\\&=&\g^{r+1}\mathfrak M_{m,1}(\lambda,L1)-
a_0\dots a_{r-1}\mathfrak M_{m,1}(\lambda,L1);\quad m=2,\dots,r.
\end{eqnarray*}
In view of the above, we arrive at the equations for the elements
of \newline $S(\mathfrak M(\lambda,L1))(t):$
\begin{equation}
\label{momevl1} \overset\cdot{\quad
S_k^{m,1}}=S_{k+r+1}^{m,1}-S_{r+1}^{1,1}S_k^{m,1};\quad
m=1,\dots,r;
\end{equation}
which can be written in the equivalent form:
\begin{equation}
\label{momintl1} S_k^{m,1}(t)=X(t)(S_k^{m,1}(0)+\int_{0}^t
X(\tau)^{-1}S_{k+r+1}^{m,1}(\tau)d\tau);
\end{equation}
where $X(t)$ is the solution of:
\begin{equation*}
\overset \cdot X(t)=-S_{r+1}^{1,1}(t)X(t); \quad X(0)=1.
\end{equation*}

 In order to find the equations for the elements of $S(\mathfrak
M(\lambda,L2))(t)$ first, using the formula $\overset  \cdot{\;
R_{\g}^1}=[R_{\g}^2,A2]$ we establish the relations similar to
\eqref{r1i0}, namely
\begin{equation}
\label{relr2}
 \overset  \cdot  {\;\;R^2_{0,0}}=R_{q+1,0}^2;\;
 \overset  \cdot  {\;\;R^2_{0,j}}=R_{q+1,j}^2-(\sum_{l=1}^j b_l)R_{0,j}^2;\quad
j=1,\dots, q-1.
\end{equation}
From the matrix equation $(\g E - L2)R_\g^2=E$ and \eqref{relr2},
acting similarly as above, we derive:
\begin{eqnarray*}
&&\overset\cdot{\quad \;\,R^2_{0,n-1}}=\overset \cdot{\quad
\mathfrak M_{1,n}}(\lambda,L2)=\\&=&\g^{q+1}\mathfrak
M_{1,n}(\lambda,L2)-\g^{q+1-n}-(\sum_{l=0}^{n-1} b_l) \mathfrak
M_{1,n}(\lambda,L2);\quad n=1,\dots,q;
\end{eqnarray*}
and using \eqref{moml1l2} we finally get
\begin{equation}
\label{momevl2} \overset \cdot {\quad
S_k^{1,n}}=S_{k+q+1}^{1,n}-S_{q+n}^{1,n}S_k^{1,n};\quad
n=1,\dots,q.
\end{equation}
Similarly to \eqref{momintl1} we find that
\begin{equation}
\label{momintl2} S_k^{1,n}(t)=Y_n(t)(S_k^{1,n}(0)+\int_{0}^t
Y_n(\tau)^{-1}S_{k+q+1}^{1,n}(\tau)d\tau);
\end{equation}
and $Y_n$ are found from:
\begin{equation*}
\overset \cdot Y_n(t)=-S_{q+n}^{1,n}(t)Y_n(t); \quad Y_n(0)=1.
\end{equation*}
Thus, in view of the above, we have obtained that if \eqref{prob}
and \eqref{sumb} have a solution, then the elements of
$S(\mathfrak M(\lambda,L1))(t)$ and $S(\mathfrak
M(\lambda,L2))(t)$ are satisfy \eqref{momevl1} and
\eqref{momevl2}.

Assuming that $(a_i(t))_{i=0}^\infty,\,(b_i(t))_{i=0}^\infty \in
l_{\infty}$ and solving the integral equations \eqref{momintl1}
and \eqref{momintl2} by iteration, one can get the following
formulas for the moments:
\begin{equation}
\label{momseries} S_k^{m,1}(t)=\frac{\displaystyle
\sum_{l=0}^{\infty}\frac{
S_{k+(r+1)l}^{m,1}(0)t^l}{l!}}{\displaystyle
\sum_{l=0}^{\infty}\frac{S_{(r+1)l}^{1,1}(0)t^l}{l!}};\quad
S_k^{1,n}(t)=\frac{\displaystyle
\sum_{l=0}^{\infty}\frac{S_{k+(q+1)l}^{1,n}(0)t^l}{l!}}{\displaystyle
\sum_{l=0}^{\infty}\frac{S_{n-1+(q+1)l}^{1,n}(0)t^l}{l!}};\qquad k
\in \mathbb{Z}_+.
\end{equation}
Comparing \eqref{momevl1} and \eqref{momevl2}, we note that all
equations \eqref{momevl1} contain the same multiplier $
S_{r+1}^{1,1}$ and this is not the case for the equations
\eqref{momevl2}, where each of the $q$ equations has its own
multiplier $S_{q+n}^{1,n}$. Further on, we will use this while
studying the finite Bogoyavlensky lattices.

For arbitrary initial data $(b_i(0))_{i=0}^{\infty} \in
l_{\infty},\; a_{i}(0),\,b_{i}(0) \ne 0, $ the local existence and
uniqueness theorem for \eqref{sumb} (in the class of bounded
solutions) was established in \cite{geh,osf} (in fact, the
lattices with matrix/operator coefficients were considered there);
for the system \eqref{prob}, it can be proved in a similar way.

In view of the above, we obtain the following integration method
of \eqref{prob}-\eqref{sumb}:
\begin{theorem}
\label{intbl} For each set of initial complex data
$(a_i(0))_{i=0}^\infty,\,(b_i(0))_{i=0}^\infty \in l_{\infty};\;
a_i(0),\,b_i(0)\ne 0,$ there exists $\delta>0$ such that the
Cauchy problem for \eqref{prob}-\eqref{sumb} has a unique solution
for $t\in [0,\delta)$ which can be found in the following way:
\begin{enumerate}
\item Construct the matrices $L1(0),\;L2(0)\,$
\eqref{probl},\eqref{sumbl} at $t=0$ out of the initial data and
find the moments $S(\mathfrak M(\lambda,L1))(0)$ and \newline
$S(\mathfrak M(\lambda,L2))(0)$. \item Calculate the moments
$S(\mathfrak M(\lambda,L1))(t)$ and $S(\mathfrak
M(\lambda,L2))(t)$ according to \eqref{momseries}. \item Using
\eqref{lefta} find the elements $a_i(t),\,b_i(t)$ for $i\in
\mathbb{Z}_+\,$ and $t\in [0,\delta),\,$ which give the required
solution.
\end{enumerate}
\end{theorem}
If $a_i(0)$ and $b_i(0)$ are all real and positive (or all
negative) then the global existence and uniqueness property for
the solutions of \eqref{prob} was established in \cite{sor}; a
similar result for the system \eqref{sumb} was obtained in
\cite{osf,akv}. As follows from the above considerations, in this
case our integration method can be applied to find the global
solution. Note that, to our knowledge, the question, for which
(complex or real) initial data there exists the bounded solution
of \eqref{prob} or \eqref{sumb} for all $t\in [0,T),\,$ remains
unclosed.
\section{Miura-like transformations}
From now on, we will assume that in \eqref{prob}-\eqref{sumb},
$\,r=q=p\,$ for some $p \ge 2$. In this case the above Weyl
matrices $\mathfrak M(\g,L1)$ and $\mathfrak M(\g,L2)$ are of the
sizes $p\times 1\,$ and $\,1\times p\,$ respectively and the
formulas \eqref{lefta} are written as follows
\begin{equation}
\label{abdet}
a_i=\frac{\Delta_{i+1}\Delta_{i-1}}{\Delta_i^2},\qquad
b_i=\frac{\Delta_{i+p}\Delta_{i-1}}{\Delta_{i+p-1}\Delta_i},
\qquad i\ge 0,
\end{equation}
where $\Delta_i\,$ are defined in \eqref{delt}.

 As
mentioned in \cite{bog}, the systems
\begin{equation*}
\overset \cdot a_i=a_i\left(\prod_{j=1}^p a_{i+j} - \prod_{j=1}^p
a_{i-j}\right),\quad i\in \mathbb{Z};
\end{equation*}
after denoting
\begin{equation}
\label{miuab} b_i=a_i\cdots a_{i+p-1},
\end{equation}
take the form
\begin{equation*}
\overset \cdot a_i=a_i(b_{i+1}-b_{i-p}),
\end{equation*}
and differentiating both sides of \eqref{miuab}, one gets
\begin{equation}
\label{probinf} \overset \cdot b_i = b_i(\sum_{j=1}^p b_{i+j} -
\sum_{j=1}^p b_{i-j}),\quad i\in\mathbb{Z}.
\end{equation}
Conversely, fix an arbitrary $i=I\,$ in \eqref{probinf}.  Then,
for $p\ge 2,$ let \newline $\,a_I,\dots, a_{I+p-2}$ be solutions
of the equations
\begin{eqnarray}
\label{invm1}
\overset \cdot a_I&=&a_I(b_{I+1}-b_{I-p}); \nonumber\\
&\vdots&\\
\overset \cdot a_{I+p-2}&=&a_{I+p-2}(b_{I+p-1}-b_{I-2}).\nonumber
\end{eqnarray}
Then, by setting
\begin{equation}
\label{aIp} a_{I+p-1}=\frac{b_I}{a_I\cdots a_{I+p-2}};
\end{equation}
one gets from \eqref{probinf}-\eqref{invm1}
\begin{eqnarray*}
&\overset \cdot a_{I+p-1}&=\frac{b_I\left(\sum_{j=1}^p
b_{I+j}-\sum_{j=1}^p b_{I-j}\right)}{a_I\cdots a_{I+p-2}}-\\
&-&\frac{b_I\left( (b_{I+1}-b_{I-p})+(b_{I+2}-b_{I-p+1})+\cdots
+(b_{I+p-1}-b_{I-2})\right)}{a_I\cdots
a_{I+p-2}}\\&=&a_{I+p-1}(b_{I+p}-b_{I-1}).
\end{eqnarray*}
In the same manner, defining successively
\begin{equation}
\label{aIp2} a_{i}:=\frac{b_{i-p+1}}{a_{i-p+1}\cdots
a_{i-1}},\quad \text{for} \quad i \ge I+p,
\end{equation}
and
\begin{equation}
\label{aIp3} a_{i}:=\frac{b_{i}}{a_{i+1}\cdots a_{i+p-1}},\quad
\text{for} \quad i \le I-1;
\end{equation}
one finds for these values of $i$ that
\begin{equation*}
\overset \cdot a_{i}=a_{i}(b_{i+1}-b_{i-p}).
\end{equation*}
The later implies that an inverse to the transformation
\eqref{miuab} can be defined according to
\eqref{invm1}-\eqref{aIp3}.

We now turn back to the Cauchy problem for semi-infinite systems
\eqref{prob}-\eqref{sumb}, considered in the previous section. The
Miura transformation, as we call it (in \cite{zh} it was called a
B\"acklund transformation; since Miura mappings may be regarded as
special cases of B\"acklund transforms, this name is  also
justified), defined by \eqref{miuab}, maps the system \eqref{prob}
to the system \eqref{sumb} with initial conditions
\begin{equation*}
b_i(0)=a_i(0)\cdots a_{i+p-1}(0).
\end{equation*}
The inverse Miura transformation for $p \ge 2$ from \eqref{sumb}
to \eqref{prob} is defined as follows:
\begin{equation}
\label{invm211} a_{i}(t)=a_{i}(0) e^{\displaystyle \int_{0}^{t}
b_{i+1}(\tau)d \tau};\quad \text{for some}\;
a_i(0)\in\mathbb{C},\,a_i(0)\ne 0;\; i=0,\dots,p-2;
\end{equation}
and for $i\ge p-1$ it is defined recurrently as
\begin{equation}
\label{invm212} a_i(t)=\frac{b_{i-p+1}(t)}{a_{i-p+1}(t)\cdots
a_{i-1}(t)}.
\end{equation}
Note that one can define the  inverse Miura transformation
starting from \eqref{invm1} and set
$a_{i}(t)=a_{i}(0)\exp(\int_0^t
(b_{i+1}(\tau)-b_{i-p}(\tau))d\tau) $ for some complex
\newline $\,a_i(0)\ne 0,\; i=I,\dots,I-p+2.\,$
 Then, using \eqref{aIp3} one arrives at the elements
 $a_0(t),\dots,a_{p-2}(t).\,$ However, if we set in
 \eqref{invm211} for $a_i(0)\,$ the values of the latter at $t=0,$ then applying
 \eqref{invm211}-\eqref{invm212}, we get the same semi-infinite system
 \eqref{prob} with the same initial data as after using the
 transformations defined according to \eqref{invm1}-\eqref{aIp3};
 in this sense, the two procedures are equivalent.

 Our next aim here is to find the expression for these Miura transformations
 in terms of the moment sequences
$S(\mathfrak M(\lambda,L1))(t)$ and $S(\mathfrak
M(\lambda,L2))(t)$ introduced in the previous section. For
convenience, denote their elements defined in
\eqref{sL1}-\eqref{sL2} as
\begin{equation*}
\begin{pmatrix}S^{1}_k\\\vdots\\S^{p}_k\end{pmatrix}
\quad \text{and}\quad (\tilde S^{1}_k,\dots,\tilde S^{p}_k),
\qquad
 k\in \mathbb{Z_+}
\end{equation*}
respectively. Note that as follows from
\eqref{sparl1}-\eqref{sparl2}
\begin{equation}
\label{sparl1l2} S_k^{l}=\tilde S_k^{l}=0, \quad \text{for} \quad
 l=1,\dots,p;\quad k \ne l-1+(p+1)k',\;k'\in \mathbb{Z}_+.
\end{equation}
\begin{theorem}
\label{thmiura} The Miura transformation \eqref{miuab} between the
systems \eqref{prob} and \eqref{sumb} for $r=q=p\ge 2\,$ can be
described as the transformation $S\to \tilde S$ between
$S:=S(\mathfrak M(\lambda,L1))(t)$ and $\tilde S :=S(\mathfrak
M(\lambda,L2))(t)$ as follows:
\begin{equation}
\label{miuS12} \frac{S_k^l(t)}{S_{l-1}^l(t)}=\tilde S_k^l(t),
\quad l=1,\dots,p; \quad k\in \mathbb{Z}_+.
\end{equation}
Conversely, the transformation \eqref{invm211}-\eqref{invm212} can
be expressed as $\tilde S \to S$ in the following way:
\begin{eqnarray}
\label{miuS21} S_{k}^1(t)=\tilde S_{k}^1(t),\quad S_{k}^l(t)&=&
a_0(0)\cdots a_{l-2}(0) e^{\displaystyle {\int_0^t (\tilde
S_{l+p}^{l}(\tau)-\tilde S_{p+1}^{1}(\tau))d\tau}}\tilde
S_{k}^l(t);
\nonumber\\
 l&=&2,\dots,p.
\end{eqnarray}
\end{theorem}
\begin{proof}
First consider \eqref{prob} and the corresponding moment sequence
$S$. In this case we have in \eqref{delt}
$\alpha_{i,j}=S_{k_1+j}^{m1}(t),\; r=p,\;$ and formula
\eqref{detform} reads
\begin{equation}
\label{dela} \Delta_k =a_{k-1}a_{k-2}^2\cdots a_0^{k},\qquad k \ge
1.
\end{equation}
For $k\ge 0,$ set
\begin{equation}
\label{hats}
\hat S_k^l(t)=\frac{S_k^l(t)}{S_{l-1}^l(t)}, \quad
l=1,\dots,p;
\end{equation}
and consider the determinants $\tilde \Delta_k =\tilde
\Delta_k(t)=\det(\hat \alpha_{i,j})_{i,j=0}^k\,$  defined in
\eqref{delt} (with $\hat \alpha_{i,j}=\hat S_{k_1+j}^{m1}(t)\,$).
Then, for $k=0,\dots, p-1,$ applying the Lemma \ref{lemmom}, we
get
\begin{equation}
\label{delini}
\tilde
\Delta_k=\frac{\Delta_k}{S_{k}^{k+1}(t)\cdots
S_0^1(t)}=\frac{a_{k-1}a_{k-2}^2\cdots
a_0^{k}}{S_{k}^{k+1}(t)\cdots S_0^1(t)}=1.
\end{equation}
For $\,k\ge p, \quad k=hp+h_1,\; h=\left\lfloor\displaystyle
\frac{k}{p}\right\rfloor\,$ we have
\begin{eqnarray}
\label{dtrans}
\tilde
\Delta_k&=&\frac{\Delta_k}{(S_{p-1}^p(t)\cdots S_0^1(t))^h
(S_{h_1}^{h1+1}(t)\cdots S_0^1(t))}=\nonumber
\\&=&\frac{a_{k-1}a_{k-2}^2\cdots a_0^{k}}{(a_{p-2}a_{p-3}^2\cdots
a_0^{p-1})^h (a_{h_1}^0\cdots a_0^{h1})}=\nonumber
\\&=&\frac{a_{k-1}a_{k-2}^2\cdots a_{h1+1}^{ph-1}(a_{h1}\cdots
a_0)^{ph}}{(a_{p-2}a_{p-3}^2\cdots
a_0^{p-1})^h}:=\frac{\text{I}}{\text{II}}
\end{eqnarray}
As in \eqref{miuab}, set $b_k=a_{k}\dots a_{k+p-1}.$ Then it can
be checked that
\begin{eqnarray}
\label{prodrep} &&b_{k-vp}\cdots b_{k-(v+1)p+1}=\nonumber
\\&=&a_{k-(v-1)p-1}a_{k-(v-1)p-2}^2\cdots a_{k-vp}^p
a_{k-vp-1}^{p-1}\cdots a_{k-(v+1)p}^2a_{k-(v+1)p+1},\nonumber\\
&& \text{for}\quad v=1,\dots,h-1.
\end{eqnarray}
Applying \eqref{prodrep}, we rearrange the numerator in the
right-hand side of \eqref{dtrans} as follows:
\begin{equation}
\label{Irep} \text{I} = (b_{k-p}\cdots b_{k-2p+1})(b_{k-2p}\cdots
b_{k-3p+1})^2 \cdots (b_{k-(h-1)p}\cdots
b_{k-hp+1})^{h-1}\times\text{III},
\end{equation}
where
\begin{equation}
\label{I3form} \text{III}=a_{p+h_1-1}^{h}a_{p+h1-2}^{2h}\cdots
a_{h_1+1}^{(p-1)h}(a_{h_1}\cdots a_0)^{ph}.
\end{equation}
Then we have from \eqref{I3form} and \eqref{dtrans}
\begin{eqnarray*}
\label{resb} &&\frac{\text{III}}{\text{II}}=\frac{a_{p+h_1-1}^h
a_{p+h_1-2}^{2h}\cdots a_{h1}^{ph}a_{h1-1}^{ph}a_{0}^{ph}}
{a_{p-2}^h a_{p-3}^{2h} a_{h1}^{(p-h_1-1)h} \cdots
a_0^{(p-1)h}}=\\&&=a_{p+h_1-1}^h a_{p+h_1-2}^{2h}\cdots
a_{p-1}^{(h_1+1)h}a_{p-2}^{(h_1+1)h}\cdots
a_{h_1}^{(h_1+1)h}a_{h_1-1}^{h_1h}\cdots a_{0}^h=(b_{h_1} \cdots
b_0)^h. \nonumber
\end{eqnarray*}
Substituting the latter formulas into \eqref{dtrans}, we finally
get
\begin{eqnarray}
\label{tildlast} \tilde \Delta_k &=& (b_{k-p}\cdots
b_{k-2p+1})(b_{k-2p}\cdots b_{k-3p+1})^2 \cdots
(b_{k-(h-1)p}\cdots b_{k-hp+1})^{h-1}\times \nonumber
\\ &\times&(b_{h_1}\cdots b_0 )^h.
\end{eqnarray}
Comparing \eqref{delini} and \eqref{tildlast} with the formulas
\eqref{detform} corresponding to the special case of matrices $L2$
(when $q=p\,$) we find that they coincide with each other. Due to
\eqref{abdet}, \eqref{sparl1l2} and the Theorem \ref{thcrit} (its
condition $(i)$ follows from \eqref{sparl1l2} and \eqref{hats};
the condition $(ii)$ follows from the fact that $(ii)$ is
fulfilled for the moment sequence $S$), this coincidence implies
that $(\hat S_1(t),\dots, \hat S_p(t))\,$ is the moment sequence
$\tilde S=S(\mathfrak M(\lambda,L2))(t)$ of the matrix $L2\,$ with
the coefficients $b_i\,$ defined by \eqref{miuab}. Thus, the
``direct'' part of the theorem is proved.

To prove the converse part, we consider the system \eqref{sumb},
the moment sequence $\tilde S\, $ and the corresponding
determinants $\tilde \Delta_k$ defined according to \eqref{delt}
and satisfying \eqref{tildlast}. Then we consider the sequence
$S=S(t)$ with the elements defined from \eqref{miuS21}. Using the
latter we set
\begin{equation*}
a_{i}(t)=\frac{S_{i+1}^{i+2}(t)}{S_{i+1}^{i+1}(t)}=a_{i}(0)e^{\displaystyle
{\int_0^t (\tilde S_{i+p+2}^{i+2}(\tau)-\tilde
S_{i+p+1}^{i+1}(\tau))d\tau}}; \quad i=0,\dots p-2.
\end{equation*}
As follows from the Lemma \ref{lemmom}, $\tilde
S_{i+p+2}^{i+2}(\tau)-\tilde
S_{i+p+1}^{i+1}(\tau)=b_{i+1}(\tau),\, $ so we arrive at the
formula \eqref{invm211}. Then we consider the determimants
$\Delta_{k}\,$ defined according to \eqref{delt}, where
$\alpha_{i,j}=S_{k_1+j}^{m1}(t)\,$ and show, reversing the
arguments used to prove the ``direct'' part and applying
\eqref{invm212}, that they satisfy \eqref{dela} (in the latter the
elements $a_k$ for $k\ge p-1\,$ are found from \eqref{invm212}).
Using this fact and the Theorem \ref{thcrit} we find that $S$ is
the moment sequence of the matrix $L1(t)$ which appears in the Lax
representation \eqref{probl}-\eqref{proba} for the system
\eqref{prob} with the elements satisfying
\eqref{invm211}-\eqref{invm212}.
\end{proof}
\section{Finite case. First integrals}
Now consider the finite systems \eqref{prob}-\eqref{sumb}, namely
\begin{eqnarray}
\overset \cdot
a_i&=&a_i\left(\prod_{j=1}^{p}a_{i+j}-\prod_{j=1}^{p}a_{i-j}\right
);
\label{probf}\\
\overset \cdot
b_i&=&b_i\left(\sum_{j=1}^{p}b_{i+j}-\sum_{j=1}^{p}b_{i-j}\right );\label{sumbf}\\
\text{where}\quad i&=&0,\dots,N; \quad a_i,b_i \in \mathbb{C},\quad t\in [0,T),\quad 0<T\le \infty;\nonumber \\
a_i,b_i &\ne& 0,\quad b_{l}=a_{l}=0  \quad\text{for} \quad l<0 \;
\text{and}\quad l>N; \nonumber
\end{eqnarray}
for a certain $p\ge 1\;$ and $ N \ge 2p-1.$ They can be called the
Bogoyavlensky lattices with open-end boundary conditions
\cite{sur}. As in the semi-infinite case, the system \eqref{probf}
admits the Lax representation with the matrix
\begin{equation*}
L=L1_N=\begin{pmatrix} 0&0&\dots& 0_{0\,p-1}&1& 0&0&\dots&0 \\
           a_0&0&0&\dots&0_{1 p}&1&0&\dots&0\\
           0&a_1&0&0&\dots&0_{2\,p+1}&1&\dots&0\\
            \vdots&\dots&\ddots&\dots&\dots&\vdots&\dots&\ddots&\vdots\\
            0&\dots&0&a_{N-1}&0&0&0&\dots&1\\
            0&\dots&0&0&a_{N}&0&0&\dots&0\\
            \end{pmatrix},
\end{equation*}
of order $N+2 \times N+2\,$ while the $L$ matrix for the system
\eqref{sumbf} takes the form
\begin{equation*}
L=L2_N
=\begin{pmatrix}0&1& 0&0&0&0&\dots&0\\
           \vdots&\dots&\ddots&\dots&\vdots&\vdots&\dots&\vdots\\
           0_{p-1\,0}&\dots&0&1_{p-1\, p}&0&0&\dots&0\\
           b_0&0_{p \,1}&\dots&0&1&0&\dots&0\\
           0&b_1&0&\dots&0&1&\dots&0\\
            \vdots&\dots&\ddots&\dots&\vdots&\dots&\ddots&\vdots\\
            0&\dots&0&b_{N-1}&0&0&\dots&1\\
            0&\dots&0&0&b_{N}&0&\dots&0\\
            \end{pmatrix};
\end{equation*}
and its order is $N+p+1\times N+p+1$. Both of them are special
cases of the matrices $M_N;$ the latter may be regarded as leading
principal submatrices of order $N+q+1$ of the above considered
matrices $M$ \eqref{matM}:
\begin{equation*}
\label{matMf}
M_N=\begin{pmatrix} a_{0 0}&\dots&1_{0\;r}&0&0&\dots&0_{0\,q+N}\\
   \vdots&\dots&\vdots&\ddots&\dots&\vdots&\vdots\\
a_{q\,0}&a_{q\,1}&a_{q 2}&\dots&1_{q\,q+r}&\dots&0\\
\vdots&\ddots&\vdots&\dots&\dots&\ddots&\vdots\\
0&\dots&a_{q+N-r\; N-r}&\dots&\dots&\dots&1_{q+N-r\,q+N}\\
\vdots&\vdots&\ddots&\dots&\dots&\dots&\vdots\\
&&&&&&\\
\vdots&\vdots&\dots&\ddots&\dots&\dots&\vdots\\
0&0&\dots&0&a_{q+N\,N}&\dots&a_{q+N\,q+N}
\end{pmatrix};
\end{equation*}
The inverse spectral problem method considered in Section 2,
including the reconstruction algorithm, is applicable as well to
$M_N$. The condition $(i)$ of the Theorem \ref{thcrit} remains the
same, while the condition $(ii)$ is replaced by:

$\qquad (ii_N)\; $ {\it For} $\; k=0,\dots, q+N $
\begin{eqnarray}
 \Delta_k &\ne& 0 \label{delfin} \\
 \quad and \quad \operatorname{rank} H&=&N+q+1, \label{rankf}
\end{eqnarray}

\hspace{58 pt}{\it where H  is  defined  in \eqref{hinf}.}
\newline Due to the Hankel type structure of matrix $H$ (for  $r=q=1$
this is an infinite Hankel matrix; as known, its subsequent
rows/columns are the ``shortened'' versions of the preceding ones,
see \cite{gan} Chapter XV Theorem 7; and this property is retained
in the general case of matrix $H$) the condition \eqref{rankf} is
equivalent to
\begin{equation}
\label{allk1} \alpha_{i,j}=\sum_{v=0}^{N+q} C_v \alpha_{i,
j-v-1},\quad C_v \in \mathbb{C}, \quad j \ge N+q+1,
\end{equation}
or
\begin{equation}
\label{allk2} \alpha_{i,j}=\sum_{v=0}^{N+q} D_v \alpha_{
i-v-1,j},\quad D_v \in \mathbb{C}, \quad i \ge N+q+1,
\end{equation}
for certain sets $C_0, \dots ,C_{N+q}$ or $D_0, \dots , D_{N+q}$
and $N+q$ is the least number for which
\eqref{allk1}-\eqref{allk2} are fulfilled. Also, if one of these
conditions is hold, then the second one is hold as well (the row
rank of a matrix is equal to its column rank). Here, we will not
give the proof of analogue of the Theorem \ref{thcrit} for the
matrices $M_N$; instead we refer to \cite{os2} where a similar
result was established for another class of finite band matrices
(see also \cite{geh}). We call $C_0, \dots ,C_{N+q}$ and $D_0,
\dots , D_{N+q}$ the finite rank coefficients (FRC) of the matrix
$H$.

As in the previous section, we denote as
\begin{eqnarray*}
&&S(\mathfrak M(\lambda,L1_N))(t) := S_N(t)=S_N
=\begin{pmatrix}S^{1}_k\\\vdots\\S^{p}_k\end{pmatrix}_{k=0}^\infty\\
\quad \text{and} \quad &&S(\mathfrak M(\lambda,L2_N))(t) :=\tilde
S_N(t)= \tilde S_N = (\tilde S_k^{1},\dots,\tilde
S_k^{p})_{k=0}^{\infty}
\end{eqnarray*}
the moment sequences of the Weyl matrices corresponding to $L1_N$
and $L2_N$ respectively. The relations \eqref{allk1} and
\eqref{allk2} for the elements of $S_N$ and $\tilde S_N$ can be
written as
\begin{eqnarray}
&&S^l_k=\sum_{v=0}^{N+1} C_v S^l_{k-v-1},\qquad k \ge N+2; \label{cs}\\
&&S^{m_l}_{k_l+k}=\sum_{v=0}^{N+1} D_v S_{\hat k_v+k}^{\hat m_v}, \quad k \ge 0; \qquad l=1,\dots,p,\label{ds}\\
\text{where}\quad
&&k_l=\left\lfloor\frac{N+1+l}{p}\right\rfloor,\quad \hat
k_v=\left\lfloor\frac{N+l-v}{p}\right\rfloor,\nonumber \\
m_l&=&\operatorname{rem}(N+1+l,p)+1, \quad \hat
m_v=\operatorname{rem}(N+l-v,p)+1; \nonumber
\end{eqnarray}
and, respectively,
\begin{eqnarray}
&&\tilde S^l_k=\sum_{v=0}^{N+p} \tilde C_v \tilde S^l_{k-v-1},\qquad k \ge N+p+1; \label{tcs}\\
&&\tilde S^{n_l}_{k_l+k}=\sum_{v=0}^{N+p} \tilde D_v \tilde S_{\hat k_v+k}^{\,\hat n_v},\quad k \ge 0; \qquad l=1,\dots,p,\label{tds}\\
\text{where}\;
k_l&=&\left\lfloor\frac{N+p+l}{p}\right\rfloor,\quad \hat
k_v=\left\lfloor\frac{N_v}{p}\right\rfloor,\; N_v=N+p+l-(v+1),\nonumber\\
n_l&=&\operatorname{rem}(N+p+l,p)+1, \quad \hat
n_v=\operatorname{rem}(N_v,p)+1; \nonumber
\end{eqnarray}

The integration procedure for the semi-infinite Bogoyavlensky
lattices considered in Section 2 can be applied as well for
integration of their finite counterparts. In particular, formulas
\eqref{momevl1} and \eqref{momevl2} are valid for the elements of
$S_N(t)$ and $\tilde S_N(t)$. Moreover, a detailed study of the
integration method in the finite case allows one to establish new
properties of the lattices under consideration; e. g. the
following result holds for the systems
\eqref{probf}-\eqref{sumbf}.
\begin{theorem}
\label{istint} The FRC $C:=\{C_0,\dots,C_{N+1}\}$ and
$D:=\{D_0,\dots,D_{N+1}\}$ introduced in \eqref{cs}-\eqref{ds} are
the integrals of motion (first integrals of) the system
\eqref{probf}, while the FRC $\tilde C :=\{\tilde C_0,\dots,
\tilde C_{N+p}\}$ defined in \eqref{tcs} are the first integrals
of the system \eqref{sumbf}.
\end{theorem}
\begin{proof}
First consider the sequence $S_N(t).$ For $S_{N+2}^1$ applying
\eqref{cs} and \eqref{momevl1} we find
\begin{eqnarray*}
\overset\cdot {\quad \, S_{N+2}^1}= \sum_{v=0}^{N+1}
\overset\cdot{\;C_v \,S_{N+1-v}^1} \,&+& \,\sum_{v=0}^{N+1}\,
\overset \cdot
C_v S_{N+1-v}^1 = \\
=\sum_{v=0}^{N+1} C_v (S_{N+1+p+1-v}^1-S_{p+1}^1S_{N+1-v}^1)&+&
\,\sum_{v=0}^{N+1}\, \overset \cdot C_v S_{N+1-v}^1 =\\
=S_{N+p+3}^1-S_{p+1}^1 S_{N+2}^1 \,+ \,\sum_{v=0}^{N+1}\, \overset
\cdot C_v S_{N+1-v}^1 &=& \overset\cdot {\quad \, S_{N+2}^1} +
\sum_{v=0}^{N+1}\, \overset \cdot C_v S_{N+1-v}^1.
\end{eqnarray*}
Thus, $$\sum_{v=0}^{N+1}\, \overset \cdot C_v S_{N+1-v}^1=0.$$
After that, applying successively this procedure $N+1$ times to
\newline
$S_{N+2}^2, \dots ,S_{N+2}^p,S_{N+3}^1, \dots, S_{\hat
k_0+N+2}^{\hat m_0}$ we arrive at the system of equations
\begin{eqnarray}
&&\overset \cdot C_0 S_{N+1}^1 \quad + \dots + \quad\overset \cdot
C_{N+1}
S_{0}^1 =0 \nonumber \\
&&\quad \vdots \qquad \qquad \quad \qquad \qquad \vdots \label{sysc}\\
&&\overset \cdot C_0 S_{\hat k_0 + N+1}^{\hat m_0} \,+ \dots
+\quad \overset \cdot C_{N+1} S_{\hat k_0}^{\hat m_0} = 0
\nonumber
\end{eqnarray}
Its determinant $\Delta_{N+1}\ne 0,$ therefore, $\overset \cdot
C_0 = \dots = \overset \cdot C_{N+1}=0$.

Then we take $S_{k_1}^{m_1}.\,$  Differentiating with respect to
$t$ both sides of the relation  $S^{m_1}_{k_1}=\sum_{v=0}^{N+1}
D_v S_{\hat k_v}^{\hat m_v}$
 and applying \eqref{momevl1}, we get
\begin{eqnarray*}
&&\overset\cdot {\quad S_{k_1}^{m_1}} = \sum_{v=0}^{N+1} D_v
(S_{\hat k_v+p+1}^{\hat m_v}-S_{p+1}^1S_{\hat k_v}^{\hat m_v}) +
\,\sum_{v=0}^{N+1}\, \overset \cdot D_v S_{\hat k_v}^{\hat m_v} =\\
&&=S_{k_1+p+1}^{m_1}-S_{p+1}^1 S_{k_1}^{m_1} \,+
\,\sum_{v=0}^{N+1}\,\overset \cdot D_v S_{\hat k_v}^{\hat m_v}\,=
\overset\cdot {\quad S_{k_1}^{m_1}} + \sum_{v=0}^{N+1}\,\overset
\cdot D_v S_{\hat k_v}^{\hat m_v},
\end{eqnarray*}
so $\sum_{v=0}^{N+1}\,\overset \cdot D_v S_{\hat k_v}^{\hat
m_v}=0$. Then applying this procedure to $S_{k_1+1}^{m_1},\dots
S_{k_1+N+1}^{m_1}$ we obtain the system
\begin{eqnarray*}
\overset \cdot D_0 S_{\hat k_0}^{\hat m_0} + &\dots& + \;\overset
\cdot D_{N+1}
S_{0}^1 =0 \\
\; &\vdots& \qquad  \vdots \\
\overset \cdot D_0 S_{\hat k_0 + N+1}^{\hat m_0}+ & \dots & +\;
\overset \cdot D_{N+1} S_{N+1}^{1} = 0.
\end{eqnarray*}
Again, its determinant $\Delta_{N+1}\ne 0,$ so $\overset \cdot D_0
= \dots = \overset \cdot D_{N+1}=0$ as well.

Further on, consider the sequence $\tilde S_N(t)$. Taking $\tilde
S_{N+p+1}^l, \, l=1,\dots,p\,$ and acting as above, we obtain
using \eqref{momevl2}
\begin{eqnarray*}
\overset\cdot {\qquad \tilde S_{N+p+1}^l}= \sum_{v=0}^{N+p}
\overset\cdot{\quad \tilde C_v \,\tilde S_{N+p-v}^{\,l}} \,&+&
\,\sum_{v=0}^{N+p}\, \overset \cdot
{\tilde C}_v \tilde S_{N+p-v}^{\,l} = \\
=\sum_{v=0}^{N+p} \tilde C_v (\tilde S_{N+2p+1-v}^{\,l}-\tilde
S_{p+1}^{\,l} \tilde S_{N+p-v}^{\,l}) &+&
\,\sum_{v=0}^{N+p}\, \overset \cdot {\tilde C}_v {\tilde S}_{N+p-v}^{\,l} =\\
=\tilde S_{N+2p+2}^{\,l}-\tilde S_{p+1}^{\,l} \tilde
S_{N+p+1}^{\,l} + \sum_{v=0}^{N+p}\, \overset \cdot {\tilde C}_v
{\tilde S}_{N+1-v}^{\,l}\; &=& \overset\cdot{\tilde
S^{\,l}}_{N+p+1}+ \sum_{v=0}^{N+p}\overset \cdot{\tilde C}_v
\tilde S_{N+p-v}^l.
\end{eqnarray*}
Therefore, $\,\sum_{v=0}^{N+p}\overset \cdot{\tilde C}_v \tilde
S_{N+p-v}^l=0\,$ for $\, l=1,\dots,p.\,$ After repeating this
procedure $\,N+1\,$ times for $\,\tilde S_{N+p+2}^1,\tilde
S_{N+p+2}^2,\dots,\tilde S_{\hat k_0+N+p+1}^{\hat n_0}\,$ we
obtain the linear homogenous system similar to \eqref{sysc} for
the unknowns $\,\overset \cdot{\tilde C}_0,\dots, \overset
\cdot{\tilde C}_{N+p}\,$ \newline with the determinant $\,\tilde
\Delta_{N+p} \ne 0,\,$ and we immediately find that \newline
$\overset \cdot{\tilde C}_0= \dots =\overset \cdot{\tilde
C}_{N+p}=0$.
\end{proof}
Remark. Unlike the case of $D,$ the above arguments are not
applicable to $\tilde D :=\{\tilde D_0,\dots,\tilde D_{N+p}\},$
due to the different structure of equations \eqref{momevl1} and
\eqref{momevl2} (see the comment after the formula
\eqref{momseries}). As we well see below, the latter, generally
speaking, are not the first integrals of \eqref{sumbf}.

 Note that a similar result was obtained in \cite{geh2}
and \cite{osrjnd} where a similar approach was applied to
integration of Volterra and Toda lattices and discrete modified
Korteweg- de Vries equation in the finite case; in these works $H$
is a Hankel matrix of finite rank.

It is known and easily verified (see e.g. \cite{bab}) that if a
finite dynamical system satisfies the Lax equation, then the
coefficients of characteristic polynomial of the corresponding $L$
matrix (and therefore its eigenvalues) are the first integrals of
the system. Below we establish the relations between the above
introduced FRC and these coefficients.
\begin{proposition}
The FPC $C$ and $\tilde C\,$ defined in the Theorem \ref{istint}
coincide up to the sign with the coefficients of characteristic
polynomials of the matrices $L1_N$ and $L2_N$ respectively.
\end{proposition}
\begin{proof}
First consider the matrix $L1_N$ and its characteristic polynomial
\begin{equation*}
P_{L1_N}(\g):=\operatorname{det}(\g I -
L1_N)=\g^{N+2}+c_0\g^{N+1}+\cdots+c_{N+1},
\end{equation*}
where $I\,$ is identity matrix. As follows from the
Cayley-Hamilton theorem,
\begin{equation}
\label{gk} L1_{N}^{N+2+k'}+c_0 L1_{N}^{N+1+k'}+ \dots + c_{N+1}
L1_{N}^{k'}=O, \quad k' \in \mathbb{Z}_+
\end{equation}
(here $O$ is a zero matrix). According to the definition,
$\;S_k^l=(L1_N^{\,k})_{l-1,0},\quad$ $\;l=0,\dots,p,\; k \in
\mathbb{Z}_{+}$. Using the equation \eqref{gk} for $k'=0,\dots,
\lfloor \frac{N+2}{2} \rfloor \,$ we get the system
\begin{eqnarray}
\label{sgk}
 S_{N+2}^1&=&-c_0 S_{N+1}^1\; - \dots \;- \; c_{N+1}S_0^1
 \nonumber\\
&&\quad \vdots \qquad \qquad \qquad \qquad \vdots \\
S_{\hat k_0+N+2}^{\hat m_0}&=&-c_0 S_{\hat k_0+N+1}^{\hat m_0} -
\dots - c_{N+1} S_{\hat k_0}^{\hat m_0};\nonumber
\end{eqnarray}
its determinant $\Delta_{N+1}\ne 0\,$. Comparing \eqref{sgk} with
\eqref{cs} we find that \newline
$\,C_0=-c_0,\dots,C_{N+1}=-c_{N+1}$. The case of $\tilde
C_{0},\dots, \tilde C_{N+p}$ and $L2_N$ is considered similarly.

\end{proof}
Thus, $C_0,\dots,C_{N+1}\,$ and $\tilde C_{0},\dots, \tilde
C_{N+p}$ may be regarded as ``standard'' first integrals of the
finite systems \eqref{probf} and \eqref{sumbf} respectively,
whereas \newline $D_0, \dots, D_{N+1},$ as non-standard first
integrals of \eqref{probf}. Note that for non-Abelian finite
Bogoyavlensky lattices (e. g. lattices with matrix elements) it
can be shown as well that the corresponding (matrix) finite rank
coefficients $C_0,\dots,C_{N+1}\,$ and $\tilde C_{0},\dots, \tilde
C_{N+p}$ are their first integrals (in \cite{geh2} a similar
result was established for non-Abelian discrete modified KdV
equation in the finite case) and they are not directly linked with
the characteristic polynomials of the corresponding Lax matrices
(the existence of possible links between these objects is, to our
knowledge, an open issue).

Now consider the Miura transformation between \eqref{probf} amd
\eqref{sumbf}. To define it correctly, see \eqref{miuab}, for the
system \eqref{sumbf} we set $N$ to be equal to $n_0$ and for
\eqref{probf} $N=n_0+p-1$ for certain $n_0 \ge 2p-1$. All findings
of the previous section, including the analog of Theorem
\ref{thmiura} for the sequences $S_N$ and $\tilde S_N,$ hold
unaltered. To prove the latter, one should take the determinants
$\Delta_{k}$ with $k$ ranging from $0$ to $n_0+p,$ rather then
$k\in \mathbb{Z}_+\,$ as in the semi-infinite case. It was
mentioned in the Introduction that such mappings transform the
first integrals of the first system into the first integrals of
the second one. Now we will see how the above results can be
helpful in studying this issue. Consider, as an example, for $p=2$
the system \eqref{probf} with $N=4:$
\begin{equation}
\label{prob4}
\begin{cases}
\overset \cdot a_0 = a_0 a_1 a_2, \qquad \overset \cdot a_1 = a_1
a_2 a_3,  \\
\overset \cdot a_2 = a_2 a_3 a_4 \,-\, a_2 a_1 a_0, \\
\overset \cdot a_3 = -a_3 a_2 a_1, \quad \; \overset \cdot a_4 =
-a_4 a_3 a_2;
\end{cases}
\end{equation}
and, respectively, system \eqref{sumbf} with $N=3:$
\begin{equation}
\label{sumb3}
\begin{cases}
\overset \cdot b_0  = b_0 (b_1+b_2),  \\
\overset \cdot b_1  = b_1 (b_2 + b_3 - b_1),  \\
\overset \cdot b_2  = b_2 (b_3 - b_1 - b_0),  \\
\overset \cdot b_3  = b_3 (-b_2 - b_1).
\end{cases}
\end{equation}
 Let $S_N$ and
$\tilde S_N$ be the corresponding moment sequences. Then, the rank
of matrix $H$ \eqref{hinf} equals to 6 for the both systems
\eqref{prob4} and  \eqref{sumb3} as well as the rank of the
matrices $L1_4$ and $L2_3$. For the elements of $S_N$ it can be
checked that in accordance with \eqref{cs}-\eqref{ds}
\begin{eqnarray*}
S_k^l&=&(a_0a_1+a_1a_2+a_2a_3+a_3a_4)S_{k-3}^l-(a_0a_1a_3a_4)S_{k-6}^l,
\quad k \ge 6,\\
S^{m_l}_{k_l+k}&=&(a_1+a_3)S_{\hat k_2+k}^{\hat
m_2}-(a_0a_3)S_{\hat k_5+k}^{\hat m_5}, \qquad \qquad l=1,2;
\end{eqnarray*}
where $ m_l, k_l, \hat m_2, \hat k_2, \hat m_5, \hat k_5 $ are
defined in \eqref{ds} (when $p=2,\;N=4$); in particular
\begin{equation}
\label{s31}
 S_3^1=(a_1+a_3)S_1^2 - (a_0 a_3)S_0^1.
\end{equation}
 Note that due to \eqref{sparl1l2}
these relations are as well hold for the zero moments. Thus, for
the matrix $L1_4$ the nonzero elements in the sets $C$ and $D$
defined in the Theorem \ref{istint} are
$\{a_0a_1+a_1a_2+a_2a_3+a_3a_4,\,-a_0a_1a_3a_4\}$ and $\{a_1+a_3,
-a_0a_3 \}$ respectively, and all of them are the first integrals
of \eqref{prob4}. As follows from the above Proposition, the
characteristic polynomial of the matrix $L1_4$ has the form
\begin{equation*}
P_{L1_4}(\g)=\g^6 -(a_0a_1+a_1a_2+a_2a_3+a_3a_4)\g^3 +
a_0a_1a_3a_4.
\end{equation*}
Note that $\sum_{i=0}^3
a_ia_{i+1}=\displaystyle{\frac{\operatorname{Tr}(L^3)}{3}}$ for
$L=L1_4$. Since $S_0^1=1,$ we find from \eqref{miuS12} that
$\tilde S_k^1=S_k^1$. Then, using the first of the equations
\eqref{tcs} we find that for the matrix $L2_3$ with the elements
$b_i$ defined from Miura mapping \eqref{miuab} the set $\tilde C$
(Theorem \ref{istint}) coincides with $ C,$ namely,
\begin{equation}
\label{tilc3}
\tilde C
=\{a_0a_1+a_1a_2+a_2a_3+a_3a_4,\,-a_0a_1a_3a_4\}=\{b_0+b_1+b_2+b_3,-b_0
b_3 \}
\end{equation}
and gives the ``standard'' first integrals of the corresponding
system \eqref{sumb3}; we denote them as $\{ I_1,I_2\}$. Moreover,
since the elements of $\tilde C$ defined in \eqref{tilc3} can be
expressed entirely via $\{b_i\}$ these are the first integrals for
the general case of \eqref{sumb3}. Using \eqref{tilc3} and the
above Proposition, we also find the characteristic polynomial of
$L2_3:$
\begin{equation*}
P_{L2_3}(\g)=\g^6 -(b_0+b_1+b_2+b_3)\g^3 + b_0 b_3.
\end{equation*}
 Now consider the pair $\{J_1,J_2\}:=\{a_1+a_3, -a_0a_3 \}\,$
of nontrivial first integrals of \eqref{prob4} from the set $D$.
Obviously, they are also the first integrals of the system
\eqref{sumb3} with the elements $b_0,\dots,b_3$ defined according
to \eqref{miuab}. As follows from the latter,
$J_2=-a_0a_3=\displaystyle{-\frac{b_0 b_2}{b_1}},$ so
$J_2=J_2(b_0,b_1,b_2)\,$ is the first integral of \eqref{sumb3} in
the general case. To find the expression of $J_1$ via $\{b_i\},$
we first consider the mapping $S_4 \to \tilde S_3,$ which is
equivalent to \eqref{miuab}, and the resulting sequence $\tilde
S_3$. As follows from \eqref{tds} and \eqref{sparl1l2}, for the
element $\tilde S_3^1$  of the latter we have
\begin{equation*}
\tilde S_3^1= \tilde D_2 \tilde S_1^2 +\tilde D_5 \tilde S_0^1,
\end{equation*}
and, as in the general case of \eqref{sumb3}, $\tilde D_2$ and
$\tilde D_5$ can be expressed in terms of $\{b_i\}$. Comparing the
latter expression with \eqref{s31} and using the Theorem
\ref{thmiura} (formula \eqref{miuS12}) we find that $\tilde D_5 =
J_2 = \displaystyle{-\frac{b_0 b_2}{b_1}}\,$ and $\,\tilde D_2 =
a_0 J_1 = a_0 a_1 +a_0 a_3 = b_0 + \displaystyle{\frac{b_0
b_2}{b_1}}$. In particular, from the latter follows that $\tilde
D_2\,$ is not the first integral of \eqref{sumb3}, see the remark
after the Theorem \ref{istint}. Thus,
$J_1=\displaystyle{\frac{b_0(b_1+b_2)}{b_1 a_0}},$ and again
applying the Theorem 3 (formula  \eqref{invm211}) we finally get
\begin{equation*}
J_1=\frac{b_0(b_1+b_2)} {b_1
a_0(0)e^{\displaystyle{\int_{0}^{t}b_1(\tau)d\tau}}},\quad
\text{for certain} \quad a_0(0)\ne 0 \in \mathbb{C}.
\end{equation*}
In view of the above, we have found that the system \eqref{sumb3}
in the general case has four integrals of motion
$\{I_1,I_2,J_1,J_2\}$ such that first two of them are the
coefficients of characteristic polynomial of the Lax matrix
corresponding to \eqref{sumb3}, and the last couple are the
``non-standard'' integrals related to \eqref{prob4}.
\section{Concluding remarks and open issues}
In view of the above, we can conclude that the description of
Miura transformation between Volterra and Toda lattices via the
inverse spectral data of the corresponding Lax operators, obtained
in \cite{osrjnd}-\cite{osconop}, can be extended  to the case of
Bogoyavlensky lattices \eqref{prob}-\eqref{sumb}. The latter, like
the former, are systems with a rich Hamiltonian structure, see
\cite{sur} Chapter 17 and \cite{zh}. For example, the finite
system \eqref{probf} can be written as
\begin{equation*}
\overset \cdot a_i = \{H_2^a, a_i\}_2^a, \qquad i=0,\dots,N;
\end{equation*}
with the Hamiltonian $H_2^a := \sum_{i} a_i a_{i+1}\,$ and the
quadratic Poisson bracket $ \{\cdot,\cdot \}_2^a \,$ defined in
the coordinates $\{a_i\}$ as follows:
\begin{equation*}
\{a_i,a_j\}_2^a = -\{a_j,a_i\}_2^a =\pi_{ij}a_i a_j, \quad i \ge
j,
\end{equation*}
where
\begin{equation*}
\pi_{ij}=\begin{cases} 0, \; i-j =0,1;\\
1,\; i-j=2,4,\dots,2\left \lfloor \frac{N}{2} \right \rfloor;\\
-1,\; i-j=3,5,\dots, 2\left \lfloor \frac{N-1}{2} \right
\rfloor+1;
\end{cases}
\end{equation*}
and the system \eqref{sumbf} admits the following Hamiltonian
representation:
\begin{equation*}
\overset \cdot b_1 = \{H_1^b,b_i\}_2^b,
\end{equation*}
whith $H_1^b:=\sum_i b_i\,$ and  Poisson bracket
$\{\cdot,\cdot\}_2^b$ with non-vanishing elements:
\begin{equation*}
\{b_{i+1},b_i\}_2^b=b_{i+1}b_i,\quad \{b_{i+2},b_i\}_2^b =
b_{i+2}b_i.
\end{equation*}
It should be noted that $\{\cdot,\cdot\}_2^b$ is a local bracket
(i. e. $\{\cdot,b_i\}_2^b$ depends only on the neighboring
coordinates $b_{i+2},b_{i+1},b_{i-1},b_{i-2}\,$,) while
$\{\cdot,\cdot \}_2^a$ is a non-local one. Obviously, the Miura
mapping \eqref{miuab} transforms $H_2^a$ to $H_1^b$ and
$\{\cdot,\cdot\}_2^a$ to $\{\cdot,\cdot\}_2^b$, so it may be
useful to apply the above results to the study of Hamiltonian
properties of Bogoyavlensky lattices.

Also, as we have seen, the equations \eqref{momevl1} and
\eqref{momevl2} are equivalent, in a certain sense, to the
original systems \eqref{prob}-\eqref{sumb}. It may be of interest
to consider \eqref{momevl1} and \eqref{momevl2} from the point of
view of the theory of integrable systems.

As known \cite{bog,zh,sur}, alongside with
\eqref{prob}-\eqref{sumb}, the family of Bogoyavlensky lattices
contains the system
\begin{equation}
\label{bog3} \overset \cdot
c_i=c_i^2\left(\prod_{j=1}^{r}c_{i+j}-\prod_{j=1}^{r}c_{i-j}\right
).
\end{equation}
The operator $L$ which appears in the Lax representation for
\eqref{bog3} differs sufficiently from the above operators $L1,L2$
and $M,\,$ and the inverse spectral problems for such operators
are less studied; some recent results in this area are contained
in \cite{osrjmp2}. The Miura transformation between \eqref{bog3}
and \eqref{sumb}
\begin{equation*}
b_i=c_ic_{i+1}\cdots c_{i+r}
\end{equation*}
was obtained by Bogoyavlensky \cite{bog}, and its description in
terms of the inverse spectral data is another interesting task.

All the above issues can be addressed for future  work.

This work is done at SRISA according to the project FNEF-2022-0007
(Reg. No 1021060909180-7-1.2.1).

\end{document}